\documentclass{amsart}%
\usepackage{amsmath}
\usepackage{amsfonts}
\usepackage{amssymb}
\usepackage{graphicx}%
\providecommand{\U}[1]{\protect\rule{.1in}{.1in}}
\newtheorem{theorem}{Theorem}

\newtheorem{corollary}[theorem]{Corollary}

\newtheorem{lemma}[theorem]{Lemma}

\newtheorem{proposition}[theorem]{Proposition}

\begin{document}
\title{Concentration of Symmetric Eigenfunctions}
\author{Daniel Azagra}
\address{Departamento de An\'{a}lisis Matem\'{a}tico\\
Universidad Complutense de Madrid\\
Fac. de CC. Matem\'{a}ticas, Avda. Complutense s/n\\
28040 Madrid, Spain}
\author{Fabricio Maci\`{a}}
\address{DEBIN\\
Universidad Polit\'{e}cnica de Madrid\\
ETSI Navales, Avda. Arco de la Victoria, s/n \\
28040 Madrid, Spain}
\email{dazagra@mat.ucm.es, fabricio.macia@upm.es}
\thanks{This work has been supported by grant Santander-Complutense 34/07-15844.}
\date{}
\maketitle

\begin{abstract}
In this article we examine the concentration and oscillation effects developed
by high-frequency eigenfunctions of the Laplace operator in a compact
Riemannian manifold. More precisely, we are interested in the structure of the
possible invariant semiclassical measures obtained as limits of Wigner
measures corresponding to eigenfunctions. These measures describe
simultaneously the concentration and oscillation effects developed by a
sequence of eigenfunctions. We present some results showing how to obtain
invariant semiclassical measures from eigenfunctions with prescribed
symmetries. As an application of these results, we give a simple proof of the
fact that in a manifold of constant positive sectional curvature, every
measure which is invariant by the geodesic flow is an invariant semiclassical measure.

\end{abstract}

\section{Introduction}

The analysis of the concentration and oscillation properties of high-frequency
solutions to Schr\"{o}dinger equations is a central theme in the study of the
correspondence principle in quantum mechanics.

Special attention has been devoted to the analysis of the high-frequency
behavior of eigenfunctions of the Laplace-Beltrami operator $\Delta_{M}$ of a
smooth compact Riemannian manifold $\left(  M,g\right)  $. The spectrum of
$-\Delta_{M}$ consists of a discrete set of eigenvalues $\left(  \lambda
_{k}\right)  $ tending to infinity. The corresponding eigenfunctions $\left(
\psi_{\lambda_{k}}\right)  $:%
\[
-\Delta_{M}\psi_{\lambda_{k}}\left(  x\right)  =\lambda_{k}\psi_{\lambda_{k}%
}\left(  x\right)  ,\quad x\in M,
\]
span the space $L^{2}\left(  M\right)  $ of square integrable functions with
respect to the Riemannian measure $dm_{g}$. In this setting, the
correspondence principle roughly asserts that high energy eigenfunctions
(\emph{i.e. }those corresponding to an eigenvalue $\lambda_{k}$ big enough)
exhibit behavior that is somehow related to the dynamics of the geodesic flow
on $\left(  M,g\right)  $ (see for instance \cite{deBievreSurv, NadTothJak,
Robert} for a more detailed account on this issue).

One associates to any eigenfunction $\psi_{\lambda_{k}}$ normalized in
$L^{2}\left(  M\right)  $ the probability distribution $\nu_{k}:=\left\vert
\psi_{\lambda_{k}}\right\vert ^{2}dm_{g}$. Then, given any sequence of
eigenvalues $\left(  \lambda_{k}\right)  $ tending to infinity, one wishes to
understand the structure of all possible limits of any sequence $\left(
\nu_{k}\right)  $ (which describe the regions on which $\psi_{\lambda_{k}}$
concentrates) and clarify how it is related to the geodesic flow in $\left(
M,g\right)  $. Instead of dealing directly with $\nu_{k}$ it is preferable to
associate to $\psi_{\lambda_{k}}$ a measure on the cotangent bundle $T^{\ast
}M$ that projects onto $\nu_{k}$. These lifts are distributions $W_{\psi
_{\lambda_{k}}}^{M}$ on $T^{\ast}M$ that act on test functions $a\in
C_{c}^{\infty}\left(  T^{\ast}M\right)  $ as:%
\[
\left\langle W_{\psi_{\lambda_{k}}}^{M},a\right\rangle =\int_{M}%
\operatorname*{op}\nolimits_{\lambda_{k}^{-1/2}}\left(  a\right)
\psi_{\lambda_{k}}\left(  x\right)  \overline{\psi_{\lambda_{k}}\left(
x\right)  }dm_{g}\left(  x\right)  .
\]
In the above formula $\operatorname*{op}\nolimits_{h}\left(  a\right)  $
stands for the semiclassical pseudodifferential operator of symbol $a$. There
is no canonical for defining $\operatorname*{op}\nolimits_{h}\left(  a\right)
$, but any two of those differ by a term which vanishes as $h\rightarrow0^{+}%
$. When $\operatorname*{op}_{h}\left(  a\right)  $ is given by the Weyl
quantization rule, the distribution $W_{\psi_{\lambda_{k}}}^{M}$ is usually
called the \emph{Wigner measure }of $\psi_{\lambda_{k}}$. More details can be
found, for instance, in \cite{Fo, Ge91c, GeLei, Li-Pau, Robert}.

Given a sequence of normalized eigenfunctions $\left(  \psi_{\lambda_{k}%
}\right)  $, the corresponding sequence $(W_{\psi_{\lambda_{k}}}^{M})$ is
bounded in the space of distributions on $T^{\ast}M$, and therefore has at
least one accumulation point. When $W_{\psi_{\lambda_{k}}}^{M}\rightharpoonup
\mu$ as $k\rightarrow\infty$ it is by now well known that the limit $\mu$ is a
Radon probability measure, supported on the unit cosphere bundle $S^{\ast}M$
that is invariant by the geodesic flow of $\left(  M,g\right)  $. In addition,
the limit of the probability densities $\nu_{k}$ may be recovered by
projecting $\mu$ onto $M$. Such a measure $\mu$ is usually called an
\emph{invariant semiclassical measure}, or a \emph{quantum limit }of $\left(
M,g\right)  $.

The problem of identifying the set of invariant semiclassical measures on a
manifold $\left(  M,g\right)  $ has attracted considerable attention in the
past thirty years. It turns out that this structure depends heavily on the
dynamical properties of the geodesic flow in $\left(  M,g\right)  $. For
instance, when it has the Anosov property (which is the case when $\left(
M,g\right)  $ has negative sectional curvature) it is known that a generic
invariant semiclassical measure (in a sense to be precised) must coincide with
the Liouville measure on $S^{\ast}M$ \cite{Shni, CdV85, Zelditch87, HMR,
GeLei, RudSar}. Moreover, it has been recently proved that quantum limits must
have positive Kolmogorov-Sinai entropy \cite{AnantAnosov, AnantNonn}. If the
geodesic flow is completely integrable, then the invariant semiclassical
measures are supported on certain invariant torii of the classical dynamics
\cite{Jak, NadTothJak, TothQInt}, which in some cases may consist on a single
geodesic \cite{JakZel, MaZoll}.

In this article, we address the question of whether an invariant semiclassical
measure $\mu$ may be realized by a sequence $\left(  \psi_{\lambda_{k}%
}\right)  $ of eigenfunctions with specified symmetries, \emph{i.e. }that are
invariant by a certain group $G$ of isometries of $\left(  M,g\right)  $. Our
main results, Theorem \ref{Thm Symm} and Corollary \ref{Cor Quot}, give a
prodecure to construct invariant semiclassical measures in a Riemannian
manifold which is the quotient of $M$ by a group of isometries that act
without fixed points. As an application of this, we show in Theorem
\ref{Thm CC}, using essentially only geometric arguments and standard
properties of spherical harmonics, that the set of invariant semiclassical
measures on a manifold $\left(  M,g\right)  $ of constant, positive sectional
curvature coincides with the whole set of probability measures on $S^{\ast}M$
that are invariant by the geodesic flow.

\medskip\noindent\textbf{Notation and conventions.} Through this article
$\left(  M,g\right)  $ will denote a smooth, compact, connected Riemannian
manifold. We shall denote the geodesic flow on $S^{\ast}M$ by $\phi_{t}^{M}$.

A generic orbit of $\phi_{t}^{M}$ in $S^{\ast}M$ will be denoted by $\gamma$;
when no confusion arises, we shall use the same notation to refer to its
projection on $M$.

Given a group $G$ of isometries of $\left(  M,g\right)  $, we shall use the
same notation to refer to the corresponding group of induced diffeomorphisms
on $S^{\ast}M$.

In what follows, we shall use the term measure to refer to a probability Radon measure.

\section{Main results}

Any group $G$ of isometries of $\left(  M,g\right)  $ defines a natural action
on the set of $\phi_{t}^{M}$-invariant measures $\mu$ in $S^{\ast}M$ by
push-forward $\phi_{\ast}\mu$, for $\phi\in G$.\footnote{The measure
$\phi_{\ast}\mu$ is defined by $\phi_{\ast}\mu\left(  \Omega\right)
:=\mu\left(  \phi^{-1}\left(  \Omega\right)  \right)  $ for every measurable
set $\Omega\subset S^{\ast}M$.} We denote the stabilizer subgroup of a measure
$\mu$ with respect to this action by:%
\[
G_{\mu}:=\left\{  \phi\in G:\phi_{\ast}\mu=\mu\right\}  .
\]
Recall that a $\phi_{t}^{M}$-invariant measure $\mu$ is ergodic if and only if
for every $\Omega\subset S^{\ast}M$ which is $\phi_{t}^{M}$-invariant,
$\mu\left(  \Omega\right)  $ must be either zero or one.

\begin{theorem}
\label{Thm Symm}Let $\left(  M,g\right)  $ be a compact Riemannian manifold
and $G$ a finite group of isometries of $M$. Let $\mu$ be a $\phi_{t}^{M}%
$-invariant ergodic measure on $S^{\ast}M$ that is an invariant semiclassical
measure realized by some sequence of $G_{\mu}$-invariant eigenfunctions. Then
\[
\left\langle \mu\right\rangle :=\frac{1}{\left\vert G\right\vert }\sum
_{\phi\in G}\phi_{\ast}\mu
\]
is an invariant semiclassical measure realized by some sequence of
$G$-invariant eigenfunctions.
\end{theorem}

Theorem \ref{Thm Symm} can be applied to show the existence of invariant
semiclassical measures on quotient Riemannian manifolds. More precisely, if
$G$ is a group of isometries that acts without fixed points on $M$ then $M/G$
is a Riemannian manifold in a natural way. Denote the natural projection by
\[
\pi:S^{\ast}M\rightarrow S^{\ast}\left(  M/G\right)  .
\]
The next result shows how invariant semiclassical measures on $M/G$ are
constructed from those on $M$.

\begin{corollary}
\label{Cor Quot}Let $\left(  M,g\right)  $ be a compact Riemannian manifold,
and $G$ be a group of isometries of $M$ that acts without fixed points. Let
$\mu$ be a $\phi_{t}^{M}$-invariant, ergodic measure in $S^{\ast}M$ that is an
invariant semiclassical measure realized by some sequence of $G_{\mu}%
$-invariant eigenfunctions in $M$. Then $\pi_{\ast}\mu$ is an invariant
semiclassical measure on the quotient Riemannian manifold $M/G$.
\end{corollary}

The proofs of Theorem \ref{Thm Symm} and Corollary \ref{Cor Quot} are
presented in Section \ref{Section Sym}. These results may be applied to
characterize the set of semiclassical invariant measures on manifolds of
positive constant sectional curvature (recall that these manifolds are
isometric to quotients of the standard sphere $\mathbb{S}^{d}$ by a group of
isometries acting without fixed points, see \cite{WolfCC}). In fact, combining
Corollary \ref{Cor Quot} with rather elementary geometric arguments, and
standard properties of spherical harmonics, we shall give in Section
\ref{Section CC} a proof of the following theorem.

\begin{theorem}
\label{Thm CC}Let $\left(  M,g\right)  $ be a Riemannian manifold of positive
constant sectional curvature. Then any $\phi_{t}^{M}$-invariant measure on
$S^{\ast}M$ is an invariant semiclassical measure.
\end{theorem}

When $\left(  M,g\right)  $ is the standard sphere or the real projective
space, Theorem \ref{Thm CC} has been proved in \cite{JakZel}. If we further
restrict ourselves to the class of homogeneous compact manifolds of constant
curvature then it turns out that those with positive curvature are precisely
the spaces having the property that the set of invariant semiclassical
measures coincides with the whole set of $\phi_{t}^{M}$-invariant measures.
This is due to the fact that there are no such spaces for $K<0$ and, when
$K=0$, such a space has to be isometric to the flat torus $\mathbb{T}^{d}$
(see for instance \cite{WolfCC}). In the latter case, a result by Bourgain
\cite{Jak} asserts that the projection on $\mathbb{T}^{d}$ of every
semiclassical invariant measure is absolutely continuous with respect to the
Lebesgue measure. Therefore, a measure supported on a geodesic cannot be an
invariant semiclassical measure.

\begin{corollary}
Let $\left(  M,g\right)  $ be a compact homogeneous Riemannian manifold of
constant sectional curvature $K$. All $\phi_{t}^{M}$-invariant measures on
$S^{\ast}M$ are invariant semiclassical measures if and only if $K>0$.
\end{corollary}

\section{\label{Section Sym}Symmetric eigenfunctions}

First of all, let us recall some of the basic properties of Wigner measures.
Let $\left(  u_{k}\right)  $ be a sequence in $L^{2}\left(  M\right)  $ such
that $W_{u_{k}}^{M}\rightharpoonup\mu$ as $k\rightarrow\infty$ for some
measure $\mu$ on $T^{\ast}M$. The following properties are well known (see for
instance \cite{Ge91c, GeLei}):%
\begin{equation}
\text{if }\phi:M\rightarrow M\text{ is a diffeomorphism then }W_{u_{k}%
\circ\phi}^{M}\rightharpoonup\phi_{\ast}\mu\text{ as }k\rightarrow
\infty.\label{diff}%
\end{equation}
Let $\left(  v_{k}\right)  $ be some other sequence such that $W_{v_{k}}%
^{M}\rightharpoonup\nu$ as $k\rightarrow\infty$; then%
\begin{equation}
\text{if }\mu\perp\nu\text{ then }W_{u_{k}+v_{k}}^{M}\rightharpoonup\mu
+\nu\text{ as }k\rightarrow\infty.\label{ort}%
\end{equation}

\begin{proof}
[Proof of Theorem \ref{Thm Symm}]Start noticing that given an isometry $\phi$,
the measure $\phi_{\ast}\mu$ is a $\phi_{t}^{M}$-invariant ergodic measure
whenever $\mu$ is. Moreover, it is not hard to see, using Birkhoff's ergodic
theorem, that either $\phi_{\ast}\mu=\mu$ or $\mu\perp\phi_{\ast}\mu$.

Indeed, for $\nu\in\left\{  \mu,\phi_{\ast}\mu\right\}  $ there exists a
measurable set $F_{\nu}\subset S^{\ast}M$ with $\nu\left(  F_{\nu}\right)  =1$
and
\[
\lim_{T\rightarrow\infty}\frac{1}{T}\int_{0}^{T}a\left(  \phi_{t}\left(
x_{0},\xi_{0}\right)  \right)  dt=\int_{S^{\ast}M}a\left(  x,\xi\right)
\nu\left(  dx,d\xi\right)  ,
\]
for every $\left(  x_{0},\xi_{0}\right)  \in F_{\nu}$ and $a\in C\left(
S^{\ast}M\right)  $. Therefore, if $F_{\mu}\cap F_{\phi_{\ast}\mu}%
\neq\emptyset$ then necessarily $\mu=\phi_{\ast}\mu$.

Now, the measures $\phi_{\ast}\mu$, $\phi\in G$, are pairwise distinct if and
only if $G_{\mu}=\left\{  \operatorname*{Id}\right\}  $. In this case, it is
easy to construct a sequence of eigenfunctions for which the conclusion holds.
Let $\left(  \psi_{\lambda_{k}}\right)  $ be such that
\begin{equation}
W_{\psi_{\lambda_{k}}}^{M}\rightharpoonup\mu,\qquad\text{as }k\rightarrow
\infty, \label{convSH}%
\end{equation}
and define the average
\[
\left\langle \psi_{\lambda_{k}}\right\rangle _{G}:=\frac{1}{\left\vert
G\right\vert }\sum_{\phi\in G}\psi_{\lambda_{k}}\circ\phi.
\]
Clearly, this is a $G$-invariant eigenfunction of $\Delta_{M}$ (that might
possibly vanish identically). Now, because of (\ref{diff}), $W_{\psi
_{\lambda_{k}}\circ\phi}^{M}$ converges to the measure $\phi_{\ast}\mu$, and
since all the measures $\phi_{\ast}\mu$ are distinct, they must be mutually
disjoint. Now, the asymptotic orthogonality property (\ref{ort}) then implies:%
\begin{equation}
W_{\left\langle \psi_{\lambda_{k}}\right\rangle _{G}}^{M}\rightharpoonup
\frac{1}{\left\vert G\right\vert }\sum_{\phi\in G}\phi_{\ast}\mu=\left\langle
\mu\right\rangle ,\qquad\text{as }k\rightarrow\infty. \label{ConvAver}%
\end{equation}
Note that, in particular, (\ref{ConvAver}) implies that the sequences of
averages $\left\langle \psi_{\lambda_{k}}\right\rangle _{G}$ is not
identically equal to zero, and therefore can be normalized in $L^{2}\left(
M\right)  $.

Suppose now that $G_{\mu}$ is non-trivial. By hypothesis, there exists
$\left(  \psi_{\lambda_{k}}\right)  $ such that (\ref{convSH}) holds and
$\psi_{\lambda_{k}}\circ\phi=\psi_{\lambda_{k}}$ for every $\phi\in G_{\mu}$.
Let $\phi_{1}=\operatorname*{Id},\phi_{2},...,\phi_{\left\vert G\right\vert
/\left\vert G_{\mu}\right\vert }$ be a common system of representatives for
the left cosets $\phi G_{\mu}$ and the right cosets $G_{\mu}\phi$ of $G_{\mu}$
in $G$ (whose existence is ensured by a classical theorem of P. Hall, see for
instance \cite{HallComb} Theorem 5.1.7). Given any $\phi\in G$, one has
$\psi_{\lambda_{k}}\circ\rho=\psi_{\lambda_{k}}\circ\phi$ for every $\rho\in
G_{\mu}\phi$; therefore, the average satisfies:%
\[
\left\langle \psi_{\lambda_{k}}\right\rangle _{G}=\frac{\left\vert G_{\mu
}\right\vert }{\left\vert G\right\vert }\sum_{l=1}^{\left\vert G\right\vert
/\left\vert G_{\mu}\right\vert }\psi_{\lambda_{k}}\circ\phi_{l}.
\]
Since there is a bijection between the orbit $\left\{  \phi_{\ast}\mu:\phi\in
G\right\}  $ and the set of left cosets of $G_{\mu}$ in $G$, all measures
$\mu,\left(  \phi_{2}\right)  _{\ast}\mu,...,\left(  \phi_{j}\right)  _{\ast
}\mu$ must be distinct. The conclusion then follows using the same argument we
gave for $\left\vert G_{\mu}\right\vert =1$.\medskip
\end{proof}

To prove Corollary \ref{Cor Quot} just take into account the following: (i)
the eigenfunctions of $\Delta_{M/G}$ are induced (via $\pi$) precisely by the
eigenfunctions of $\Delta_{M}$ that are $G$-invariant; (ii) given any measure
$\mu$ in $M$, $\pi_{\ast}\mu=\pi_{\ast}\left\langle \mu\right\rangle $ and

\begin{lemma}
\label{LemmaWignerComp}It is possible to give a definition of Wigner measures
in $M$ and $M/G$ such that for every $u\in L^{2}\left(  M\right)  $ which is
$G$-invariant, one has that if $W_{u}^{M}$ is the Wigner measure of $u$ in $M$
then $W_{u}^{M/G}=\pi_{\ast}W_{u}^{M}$ is the Wigner measure of $u$ in $M/G$.
\end{lemma}

The proof of this result follows the classical construction via local charts
(see, for instance \cite{MaZ}, Section 3); it suffices to construct the Wigner
measures from an atlas in $\left(  U_{i},\varphi_{i}\right)  $, $i=1,..r$, in
$M/G$ and an atlas $\left(  V_{i,j},\tilde{\varphi}_{i,j}\right)  $ in $M$
such that $V_{i,j}\subset\pi^{-1}\left(  U_{i}\right)  $, and $\tilde{\varphi
}_{i,j}=\pi\circ\varphi_{i}$, where $\pi:M\rightarrow M/G$ is the natural
projection. We emphasize the fact that the set of invariant semiclassical
measures on a manifold $\left(  M,g\right)  $ does not depend of the notion of
Wigner measure used to realize it (see, for instance, \cite{GeLei, MaZ}).

\section{\label{Section CC}Positive sectional curvature}

We now turn to analyze the structure of invariant semiclassical measures in
manifolds of constant, positive sectional curvature. Any such manifold
$\left(  M,g\right)  $ is the quotient of $\mathbb{S}^{d}$ by a group $G$ of
isometries that acts without fixed points. Recall that the eigenvalues of
$-\Delta_{\mathbb{S}^{d}}$ are $\lambda_{k}=k\left(  k+d-1\right)  $ and the
corresponding eigenfunctions $\psi_{k}$ are spherical harmonics of degree $k$;
the eigenfunctions of $-\Delta_{M}$ are precisely those spherical harmonics
that are $G$-invariant.

Since every $\phi_{t}^{M}$-invariant measure in $S^{\ast}M$ may be
approximated by finite convex combinations of measures $\delta_{\gamma}$ with
$\gamma$ a geodesic (by the Krein-Milman theorem), Theorem \ref{Thm CC} is
then a consequence of Lemma \ref{LemmaWignerComp} and of the following result.

\begin{proposition}
\label{Proposition inter}Let $G$ be a group of order $p$ formed by isometries
of $\mathbb{S}^{d}$ acting without fixed points. Given any geodesic $\gamma$
on $S^{\ast}\mathbb{S}^{d}$ there exist a sequence $\left(  \psi_{kp}\right)
$ of normalized, $G$-invariant spherical harmonics such that $\delta_{\gamma}$
is an invariant semiclassical measure realized by $\left(  \psi_{kp}\right)  $.
\end{proposition}

Note that, in particular, this shows that $kp\left(  kp+d-1\right)  $ are
eigenvalues of $-\Delta_{M}$; more detailed results on the structure of the
spectrum of manifolds of constant, positive sectional curvature may be found
in \cite{IkedaSp, IkedaSp2, Prufer1, Prufer2, Prufer3}.

As a consequence of Theorem \ref{Thm Symm}, the proof of Proposition
\ref{Proposition inter} may be reduced to that of the following simpler result.

\begin{proposition}
\label{Prop refinitiva}Let $G$ and $p$ be as above. Given any geodesic
$\gamma$ in $S^{\ast}\mathbb{S}^{d}$ there exists a sequence $\left(
\psi_{kp}\right)  $ of normalized $G_{\gamma}$-invariant spherical harmonics
such that $W_{\psi_{kp}}^{\mathbb{S}^{d}}\rightharpoonup\delta_{\gamma}$ as
$k\rightarrow\infty$, where $G_{\gamma}$ is the subgroup of $G$ consisting of
the $\phi\in G$ such that $\phi\left(  \gamma\right)  =\gamma$.
\end{proposition}

Note that Proposition \ref{Prop refinitiva} is a direct consequence of Theorem
1 in \cite{JakZel} when $d$ is even, since in this case either $G=\left\{
\operatorname*{Id}\right\}  $ or $G=\left\{  \operatorname*{Id}%
,-\operatorname*{Id}\right\}  $, see \cite{WolfCC}. Therefore, we shall assume
in what follows that $d$ is odd, and therefore $\mathbb{S}^{d}$ is contained
in an even-dimensional euclidean space $\mathbb{R}^{d+1}$. Write $n:=\left(
d+1\right)  /2$, in what follows, we shall identify $\mathbb{R}^{d+1}$ to
$\mathbb{C}^{n}$. The isometries of $\mathbb{S}^{d}$ that act without fixed
points are restrictions to $\mathbb{S}^{d}$ of maps belonging to $SO\left(
d+1\right)  $. Given any $\phi\in SO\left(  d+1\right)  $, there exist
$\varphi\in SO\left(  d+1\right)  $ and $\left(  \theta_{1},...,\theta
_{n}\right)  \in\mathbb{T}^{n}$ such that%
\begin{equation}
\varphi^{-1}\phi\varphi=\left[
\begin{array}
[c]{ccc}%
e^{i\theta_{1}} & \cdots & 0\\
\vdots & \ddots & \vdots\\
0 & \cdots & e^{i\theta_{n}}%
\end{array}
\right]  . \label{canonical}%
\end{equation}
Our next result clarifies the structure of the groups of isometries that leave
a geodesic invariant.

\begin{lemma}
\label{Lemma cyclic}Let $H\subset SO\left(  d+1\right)  $ be a finite subgroup
that acts without fixed points on $\mathbb{S}^{d}$. If $H$ leaves a geodesic
$\gamma$ in $\mathbb{S}^{d}$ invariant then $H$ must be cyclic.
\end{lemma}

\begin{proof}
Suppose $\gamma$ is obtained as the intersection of $\mathbb{S}^{d}$ with a
plane $\pi_{\gamma}\subset\mathbb{R}^{d+1}$ through the origin. Then every
$\phi\in H$ leaves invariant both $\pi_{\gamma}$ and $\left(  \pi_{\gamma
}\right)  ^{\perp}$. Therefore, there exists a $\varphi\in SO\left(
d+1\right)  $ such that every $\phi\in G$ is of the form:%
\[
\phi=\varphi^{-1}\left[
\begin{array}
[c]{cc}%
e^{i\theta} & 0\\
0 & Q
\end{array}
\right]  \varphi,
\]
for some $\theta\in\mathbb{S}^{1}$, $Q\in SO\left(  d-1\right)  $. Since
$\phi$ has no fixed points, the order of $e^{i\theta}$ must coincide with the
order of $\phi$ and must divide $p:=\left\vert H\right\vert $. By the same
reason, $H$ cannot have elements of the form%
\[
\varphi^{-1}\left[
\begin{array}
[c]{cc}%
\operatorname*{Id} & 0\\
0 & Q
\end{array}
\right]  \varphi,\qquad Q\in SO\left(  d-1\right)  ;
\]
unless $Q=\operatorname*{Id}$. This shows that $H$ is conjugate in $SO\left(
d+1\right)  $ to a subgroup of the group consisting of the elements%
\[
\left[
\begin{array}
[c]{cc}%
e^{2\pi ij/p} & 0\\
0 & Q
\end{array}
\right]  ,\qquad j=1,...,p,\quad Q\in SO\left(  d-1\right)  ,
\]
which is isomorphic to $\mathbb{Z}_{p}\times SO\left(  d-1\right)  $. But any
subgroup $A$ of $\mathbb{Z}_{p}\times SO\left(  d-1\right)  $ having the
property that the identity is the only element of the form $\left(
0,Q\right)  $ must necessarily be cyclic. 

Indeed, let $C\subset\mathbb{Z}_{p}$ be the (cyclic) subgroup consisting of
the $q\in\mathbb{Z}_{p}$ such that $\left(  q,h\right)  \in A$ for some $h\in
SO\left(  d-1\right)  $. Given any $q\in C$, there exists a unique $h\in
SO\left(  d-1\right)  $ such that $\left(  q,h\right)  \in G$ (otherwise,
there would exist elements in $A$ of the form $\left(  0,h\right)  $ with
$h\neq\operatorname*{Id}$); denote it by $\chi\left(  q\right)  $. Clearly,
the mapping $\chi:C\rightarrow SO\left(  d-1\right)  $ is an injective group
homomorphism and $A$ is the graph of $\chi$. If $q_{0}$ is a generator of $C$
then necessarily $\left(  q_{0},\chi\left(  q_{0}\right)  \right)  $ is a
generator of $A$.
\end{proof}

\medskip

\begin{proof}
[Proof of Proposition \ref{Prop refinitiva}]Let $p\in\mathbb{N}$ and take
$l_{1},...,l_{n}$ positive integers less than or equal to $p$ and coprime with
$p$. Denote by $G\left(  p,l_{1},...,l_{n}\right)  $ the subgroup of
$SO\left(  d+1\right)  $ generated by%
\begin{equation}
\phi:=\left[
\begin{array}
[c]{ccc}%
e^{2\pi il_{1}/p} & \cdots & 0\\
\vdots & \ddots & \vdots\\
0 & \cdots & e^{2\pi il_{n}/p}%
\end{array}
\right]  ; \label{gen}%
\end{equation}
which acts without fixed points. Let $\gamma$ be a geodesic in $S^{\ast
}\mathbb{S}^{d}$ and write $G_{\gamma}:=G_{\delta_{\gamma}}$. Clearly,
$G_{\gamma}$ is the subgroup of $G$ consisting of the isometries that leave
$\gamma$ invariant.\medskip\ 

\noindent\emph{It suffices to prove the conclusion for }$G_{\gamma}=G\left(
p,l_{1},...,l_{n}\right)  $\emph{.} This is due to the fact that any subgroup
$G_{\gamma}$ is generated by some element $\rho$ of order $p:=\left\vert
G_{\gamma}\right\vert $, as a consequence of Lemma \ref{Lemma cyclic}. Now,
since $\rho$ is conjugate in $SO\left(  d+1\right)  $ to an element of the
form (\ref{canonical}) and $\rho^{k}$ has no fixed points for $1\leq k<p$ we
conclude that $G_{\gamma}=\varphi^{-1}G\left(  p,l_{1},...,l_{n}\right)
\varphi$ for some $\varphi\in SO\left(  d+1\right)  $ and some positive
integers $l_{1},...,l_{r}\leq p$ coprime with $p$. Let $\tilde{\gamma
}:=\varphi^{-1}\left(  \gamma\right)  $; the geodesic $\tilde{\gamma}$ clearly
satisfies $G_{\tilde{\gamma}}=G\left(  p,l_{1},...,l_{n}\right)  $; if
$\delta_{\tilde{\gamma}}$ is an invariant semiclassical measure realized by a
sequence $(\tilde{\psi}_{kp})$ of $G_{\tilde{\gamma}}$-invariant spherical
harmonics then $\psi_{kp}:=\tilde{\psi}_{kp}\circ\varphi$ is a sequence of
$G_{\gamma}$-invariant spherical harmonics satisfying (because of
(\ref{diff})) $W_{\psi_{kp}}^{\mathbb{S}^{d}}\rightharpoonup\delta
_{\varphi\left(  \tilde{\gamma}\right)  }=\delta_{\gamma}$ as $k\rightarrow
\infty$.\medskip\ 

\noindent\emph{The conclusion holds for }$G_{\gamma}=G\left(  p,l_{1}%
,...,l_{n}\right)  $. Start considering the geodesics $\gamma_{j}$ defined by
$\left\vert x_{2j-1}\right\vert ^{2}+\left\vert x_{2j}\right\vert ^{2}=1$.
Let
\[
\psi_{k}^{0}\left(  x\right)  :=C_{k}\left(  x_{2j-1}+ix_{2j}\right)  ^{k},
\]
where $C_{k}>0$ is chosen to have $\left\Vert \psi_{k}^{0}\right\Vert
_{L^{2}\left(  \mathbb{S}^{d}\right)  }=1$. Clearly, $\psi_{k}^{0}$ is a
spherical harmonic of degree $k$ and, as is well known (see for instance
\cite{JakZel}), $W_{\psi_{k}^{0}}^{\mathbb{S}^{d}}\rightharpoonup
\delta_{\gamma_{j}}$ as $k\rightarrow\infty$. Moreover, it easy to check that%
\[
\psi_{k}^{0}\left(  \phi\left(  x\right)  \right)  =C_{k}e^{2\pi ikl_{j}%
/p}\left(  x_{2j-1}+ix_{2j}\right)  ^{k};
\]
and in particular $\psi_{kp}^{0}\circ\phi=\psi_{kp}^{0}$. Therefore
$\delta_{\gamma_{j}}$ is a $G_{\gamma}$-invariant semiclassical measure
realized by the sequence $(\psi_{kp}^{0})$.

Let now $\gamma$ be a $\phi$-invariant geodesic in $\mathbb{S}^{d}$. Suppose
that there exists a $\chi\in SO\left(  d+1\right)  $, commuting with $\phi,$
and such that $\chi\left(  \gamma_{j}\right)  =\gamma$; clearly $\psi
_{kp}:=\psi_{kp}^{0}\circ\chi$ is again $\phi$-invariant and $\delta_{\gamma}$
is an invariant semiclassical measure realized by $\left(  \psi_{kp}\right)
$. The proof will be concluded as soon as we show that such a $\chi$ exists.
If the geodesic $\gamma$ differs from the $\gamma_{j}$ then it must be the
intersection of $\mathbb{S}^{d}$ with a plane $\pi_{\gamma}\subset
\mathbb{R}^{d+1}$ which is also a complex line in $\mathbb{C}^{n}$ and that is
invariant by $\phi$. Therefore, it must be contained in a linear subspace
$E_{\gamma}$ of $\mathbb{C}^{n}$ on which $\phi$ acts as multiplication by
some fixed $e^{2\pi il_{j}/p}$. Define $\chi$ as an element of $SU\left(
n\right)  \subset SO\left(  d+1\right)  $ such that $\chi\left(  \gamma
_{j}\right)  =\gamma$ and $\chi$ is the identity on the orthogonal of
$E_{\gamma}$. Clearly, $\chi_{|E_{\gamma}}$ commutes with multiplication by
$e^{2\pi il_{j}/p}$ and the result follows.
\end{proof}

Add some references: Berry at les houches, buscar referencias f\'{\i}sicas
(voros, heller...)

\end{document}